\documentclass[10pt, a4paper]{amsart}
\usepackage[dvips,final]{graphics}
\usepackage{array}
\usepackage{arydshln}
\usepackage[makeroom]{cancel}
 \usepackage[all]{xy}
 \usepackage{url}
\usepackage{multirow, blkarray}
\usepackage{booktabs}
\usepackage{textcomp}
\usepackage[final]{epsfig}
\usepackage{color}
\usepackage[T1]{fontenc}      
\usepackage[english,french]{babel}
\usepackage[utf8]{inputenc}
\usepackage{blindtext}

\usepackage{amsfonts,amscd,array, mathdots, epigraph}
\usepackage{amsmath}
\usepackage{amssymb}
\usepackage{amsthm}
\usepackage{mathrsfs}
\usepackage{stmaryrd}
\usepackage{diagbox}
\usepackage{enumitem}
\usepackage{ulem}

\usepackage{ulem}
\usepackage{tikz}
\usepackage{xcolor}
\usepackage{multicol}

\usepackage{fullpage}

\newtheorem{theorem}{Theorem}[section]

\newtheorem{proposition}[theorem]{Proposition}
\newtheorem{corollary}[theorem]{Corollary}

\newtheorem*{pro}{Problem}
\newtheorem*{examples}{Examples}
\newtheorem*{example}{Example}
\newtheorem{lemma}[theorem]{Lemma}
\newtheorem{definition}[theorem]{Definition}
\newtheorem*{remark}{Remark}

\numberwithin{equation}{section}

\newcommand{\bZ}{\mathbb{Z}}

\title{Finiteness of the number of irreducible $\lambda$-quiddities over a finite commutative and unitary ring}

\author{Flavien Mabilat}

\date{}
\subjclass[2020]{05E16, 05A20, 20H05}

\keywords{$\lambda$-quiddity; modular group; finite rings, Coxeter's friezes}

\email{flavien.mabilat@univ-reims.fr}

\begin{document}

\maketitle

\selectlanguage{french}
\begin{abstract}
Les $\lambda$-quiddités de taille $n$ sont des $n$-uplets d'éléments d'un ensemble fixé, solutions d'une équation matricielle apparaissant lors de l'étude des frises de Coxeter. L'étude de ces solutions passe notamment par l'utilisation d'une notion d'irréductibilité. L'objectif principal de ce texte est de montrer qu'il y a un nombre fini de $\lambda$-quiddités irréductibles lorsque l'on se place sur un anneau commutatif unitaire fini et d'obtenir dans ce cas une majoration de la taille maximale de ces dernières.

\end{abstract}

\selectlanguage{english}
\begin{abstract}
A $\lambda$-quiddity of size $n$ is an $n$-tuple of elements from a fixed set, which is a solution to a matrix equation that arises in the study of Coxeter's friezes. The study of these solutions involves in particular the use of a notion of irreducibility. The main objective of this text is to demonstrate that there is a finite number of irreducible $\lambda$-quiddities over a finite unitary commutative ring and to obtain in this case an upper bound for their maximal size.
\end{abstract}

\selectlanguage{french}

\ \\
\begin{flushright}
\textit{\og Les choses n'arrivent quasi jamais comme on se les imagine. \fg} 
\\Madame de Sévigné, \textit{Lettres.}
\end{flushright}
\ \\

\selectlanguage{english}

\section{Introduction}
\label{Intro}

Since its appearance at the beginning of the 1970s, the concept of Coxeter's friezes (\cite{Cox}) has attracted significant attention. Indeed, these objects have been the source of many beautiful results and are closely linked to many topics (see for example \cite{Mo1}). One of the most important elements of the study of Coxeter's friezes over a subset $R$ of a commutative and unitary ring $A$ is the resolution of the following equation over $R$ :
\[M_{n}(a_1,\ldots,a_n):=\begin{pmatrix}
    a_{n} & -1_{A} \\[4pt]
     1_{A} & 0_{A}
    \end{pmatrix}
\begin{pmatrix}
    a_{n-1} & -1_{A} \\[4pt]
     1_{A} & 0_{A}
    \end{pmatrix}
    \cdots
    \begin{pmatrix}
    a_{1} & -1_{A} \\[4pt]
     1_{A} & 0_{A}
     \end{pmatrix}=-Id.\]
		
\noindent Besides, the presence of the matrices $M_{n}(a_1,\ldots,a_n)$ is particularly interesting, as they also appear in the study of a lot of other mathematical objects, such as Hirzebruch-Jung continued fractions or discrete Sturm-Liouville equations (see for instance the introduction of \cite{O}).
\\
\\ \indent Moreover, the study of the previous equation naturally leads to consider the generalized equation below over a subset $R$ of a commutative and unitary ring $A \neq \{0_{A}\}$ :
\begin{equation}
\label{p}
\tag{$E_{R}$}
M_{n}(a_1,\ldots,a_n)=\pm Id.
\end{equation}
Let $n \in \mathbb{N}^{*}$, we will say that a solution $(a_1,\ldots,a_n)$ of \eqref{p} is a $\lambda$-quiddity of size $n$ over $R$ (note that in this text $\mathbb{N}$ is the set of all non negative integers and $\mathbb{N}^{*}$ is the set of all positive integers). Our objective is to obtain information about $\lambda$-quiddities over different sets. There are several ways to achieve this goal. The first way, and the most natural, is looking for a general construction of all the solutions of \eqref{p}. In this direction, we have, for example, a recursive construction and a combinatorial description of all the solutions of $(E_{\mathbb{N}^{*}})$ (see \cite{O} Theorem 1 and Theorem 2). Another possibility is to find some general information, such as the number of $\lambda$-quiddities of fixed size. Several results of that kind are known (see \cite{CO,CM,M5,Mo2}). However, here, we will consider a third way. Indeed, to study the solutions of \eqref{p} it is convenient to introduce a notion of irreducible $\lambda$-quiddity which allow us to restrict our attention to these solutions only (see \cite{C} and the next section).
\\
\\ \indent A lot of results concerning irreducible $\lambda$-quiddities are already known. In particular, we have a complete list of the irreducible solutions of \eqref{p} over several sets (see Section \ref{chap22}). However, in most cases, we do not have such a list and we try in a more modest way to find information about irreducible solutions. In particular, two main questions arise : Is the number of irreducible $\lambda$-quiddities over $R$ finite ? Is the size of irreducible solutions of \eqref{p} bounded ? If $R$ is finite then the two questions are essentially the same and this is precisely the case we want to deal with here. In particular, we have conjectured in a previous text that, for all $N \geq 2$, $(E_{\mathbb{Z}/N\mathbb{Z}})$ has a finite number of irreducible solutions (\cite{M1} Conjecture 1). Note that the resolution of \eqref{p} with $R=\mathbb{Z}/N\mathbb{Z}$ is closely linked to the different writings of the elements of the congruence subgroup $\hat{\Gamma}(N):=\{C \in SL_{2}(\mathbb{Z}),~C= \pm Id~[N]\}.$ Indeed, all the matrices of $SL_{2}(\bZ)$ can be written in the form $M_{n}(a_1,\ldots,a_n)$, with $a_{i} \in \mathbb{N}^{*}$. Since this expression is not unique, it is natural to look for all the writings of this form for a given matrix, or a set of matrices.
\\
\\ \indent Our main objective in this text is to show that the number of irreducible $\lambda$-quiddities over a finite commutative and unitary ring $A$ is finite, which in particular will give us a solution to the conjecture mentioned in the precedent paragraph. Besides, we want to have an upper bound of the size of irreducible $\lambda$-quiddities over $A$. For a ring $A$, we set ${\rm car}(A)$ the characteristic of $A$ and we recall the classical following property: if $A$ is finite then ${\rm car}(A) \neq 0$. In this article, we will prove the following results :

\begin{theorem}
\label{prin1}

Let $(A,+,\times)$ be a finite commutative and unitary ring (different from $\{0_{A}\}$) and $R \subset A$ a submagma of $(A,+)$, that is to say a subpart of $A$ closed under $+$.
\\
\\i) There is a finite number of irreducible $\lambda$-quiddities over $R$.
\\
\\ii) Let $\ell_{A}$ be the maximum size of an irreducible $\lambda$-quiddity over $A$. 
\begin{itemize}
\item If $car(A)=2$ then $4 \leq \ell_{A} \leq \frac{\left|SL_{2}(A)\right|}{\left|A\right|}+2$.
\item If $car(A) \neq 2$ then ${\rm max}(4,car(A)) \leq \ell_{A} \leq \frac{\left|SL_{2}(A)\right|}{2\left|A\right|}+2$.
\end{itemize}

\end{theorem}

\begin{corollary}
\label{prin2}

i) Let $N \geq 5$ and $p_{1},\ldots,p_{r}$ the distinct prime factors of $N$. 
\[N \leq \ell_{\bZ/N\bZ} \leq \frac{N^{2}}{2}\prod_{1=1}^{r}\left(1-\frac{1}{p_{i}^{2}}\right)+2.\]

\noindent ii) Let $q$ be the power of a prime number $p$.
\begin{itemize}
\item If $p=2$ then $4 \leq \ell_{\mathbb{F}_{q}} \leq q^{2}+1$.
\item If $p \neq 2$ then ${\rm max}(4,p) \leq \ell_{\mathbb{F}_{q}} \leq \frac{q^{2}+3}{2}$.
\end{itemize}

\end{corollary}

\noindent Theorem \ref{prin1} is proved in Section \ref{chap31} while the proof of Corollary \ref{prin2} is given in Section \ref{chap32}.

\section{Definitions and preliminary results}
\label{def}

The aim of this section is to provide the precise definitions of the notions mentioned in the introduction and to provide some elements which will be useful in the proofs of our main results. To have a complete understanding of irreducible $\lambda$-quiddities, we will also recall some known results obtained for several specific rings. Throughout this section, $(A,+,\times)$ is a commutative and unitary ring different from $\{0_{A}\}$ and $R$ is a submagma of $(A,+)$. Let $0_{A}$ denote the identity element of $+$ and $1_{A}$ the identity element of $\times$. If $N \geq 2$ and $a \in \mathbb{Z}$, we define $\overline{a}:=a+N\mathbb{Z}$.

\subsection{Quiddities}
\label{chap21}

\noindent We begin by the precise definition of the concept of $\lambda$-quiddity.

\begin{definition}[\cite{C}, Definition 2.2]
\label{21}

Let $n \in \mathbb{N}^{*}$. The $n$-tuple $(a_{1},\ldots,a_{n}) \in R^{n}$ is a $\lambda$-quiddity over $R$ if $(a_{1},\ldots,a_{n})$ is a solution of \eqref{p}, that is to say if there exists $\epsilon \in \{\pm 1_{A}\}$ such that $M_{n}(a_{1},\ldots,a_{n})=\epsilon Id$. If there is no ambiguity, we will omit to precise the ring. 

\end{definition}

Throughout the rest of this text, we will speak indiscriminately of the solutions of \eqref{p} or of the $\lambda$-quiddities over $R$. To define the notion of irreducibility, we need the two following definitions.

\begin{definition}[\cite{C}, Lemma 2.7]
\label{22}

Let $(a_{1},\ldots,a_{n}) \in R^{n}$ and $(b_{1},\ldots,b_{m}) \in R^{m}$. We define the following operation : \[(a_{1},\ldots,a_{n}) \oplus (b_{1},\ldots,b_{m})= (a_{1}+b_{m},a_{2},\ldots,a_{n-1},a_{n}+b_{1},b_{2},\ldots,b_{m-1}).\] The $(n+m-2)$-tuple obtained is the sum of $(a_{1},\ldots,a_{n})$ with $(b_{1},\ldots,b_{m})$.

\end{definition}

\begin{examples}

{\rm We give here some examples of sums over $A=R=\mathbb{Z}$ :
\begin{itemize}
\item $(1,2,3) \oplus (1,0,-2,4)=(5,2,4,0,-2)$;
\item $(-1,1,0,2) \oplus (2,2,2)=(1,1,0,4,2)$;
\item $(2,1,-,1,0,-3) \oplus (2,3,1,1)=(3,1,-1,0,-1,3,1)$.
\end{itemize}
}
\end{examples}

The operation $\oplus$ is very useful since it has the interesting following property. Let $(b_{1},\ldots,b_{m})$ be a solution of \eqref{p}. $(a_{1},\ldots,a_{n}) \oplus (b_{1},\ldots,b_{m})$ is a $\lambda$-quiddity over $R$ if and only if $(a_{1},\ldots,a_{n})$ is a solution of \eqref{p} (see \cite{C} Lemma 2.7 and \cite{WZ} Lemma 1.9). Besides, if $0_{A} \in R$, we have the following equality : 
\[(a_{1},\ldots,a_{n}) \oplus (0_{A},0_{A})=(0_{A},0_{A}) \oplus (a_{1},\ldots,a_{n})=(a_{1},\ldots,a_{n}).\]
\noindent However, $\oplus$ is neither commutative nor associative (see \cite{WZ} example 2.1).

\begin{definition}[\cite{C}, Definition 2.5]
\label{23}

Let $(a_{1},\ldots,a_{n}) \in R^{n}$ and $(b_{1},\ldots,b_{n}) \in R^{n}$. The $n$-tuple $(a_{1},\ldots,a_{n})$ is said to be equivalent to $(b_{1},\ldots,b_{n})$ (denoted by $(a_{1},\ldots,a_{n}) \sim (b_{1},\ldots,b_{n})$) if $(b_{1},\ldots,b_{n})$ can be obtained by cyclic rotations of $(a_{1},\ldots,a_{n})$ or of $(a_{n},\ldots,a_{1})$.

\end{definition}

The relation $\sim$ is an equivalence relation on the set $R^{n}$ (see \cite{WZ} Lemma 1.7). Moreover, if $(a_{1},\ldots,a_{n}) \sim (b_{1},\ldots,b_{n})$ then $(a_{1},\ldots,a_{n})$ is a solution of \eqref{p} if and only if $(b_{1},\ldots,b_{n})$ is a $\lambda$-quiddity over $R$ (see \cite{C} Proposition 2.6).
\\
\\We can now define the notion of irreducibility mentioned in the introduction.

\begin{definition}[\cite{C}, Definition 2.9]
\label{24}

A $\lambda$-quiddity over $R$ $(c_{1},\ldots,c_{n})$ ($n \geq 3$) is said to be reducible if there exists a $\lambda$-quiddity over $R$ $(b_{1},\ldots,b_{l})$ and an $m$-tuple $(a_{1},\ldots,a_{m}) \in R^{m}$ such that \begin{itemize}
\item $(c_{1},\ldots,c_{n}) \sim (a_{1},\ldots,a_{m}) \oplus (b_{1},\ldots,b_{l})$;
\item $m \geq 3$ and $l \geq 3$.
\end{itemize}
A $\lambda$-quiddity over $R$ is said to be irreducible if it is not reducible.

\end{definition}

\begin{example}

{\rm We consider here $\lambda$-quiddities over $A=R=\mathbb{Z}/9\mathbb{Z}$. $(\overline{3},\overline{3},\overline{3},\overline{3},\overline{3},\overline{3})$ is a reducible solution over $A$. Indeed, $(\overline{6},\overline{3},\overline{3},\overline{6})$ is a $\lambda$-quiddity over $A$ and $(\overline{3},\overline{3},\overline{3},\overline{3},\overline{3},\overline{3})=(\overline{6},\overline{3},\overline{3},\overline{6}) \oplus (\overline{6},\overline{3},\overline{3},\overline{6})$.
}
\end{example}

\begin{remark} 

{\rm We suppose $0_{A} \in R$. $(0_{A},0_{A})$ is never considered as an irreducible solution of \eqref{p}.}

\end{remark}

We define $\ell_{R} \in \mathbb{N} \cup \{+\infty\}$ as the least upper bound of the size of irreducible $\lambda$-quiddities over $R$ (with $\ell_{R}=0$ if $(E_{R})$ has no solution). In the next two subparts, we will give some preliminary information about $\ell_{R}$.

\subsection{Some classification results}
\label{chap22}

The objective of this subsection is to gather some classification results already known concerning irreducible $\lambda$-quiddities. We begin by the case $A=\mathbb{Z}$.

\begin{theorem}[\cite{C}, Theorems 3.1 and 3.2]
\label{25}

i) The set of irreducible $\lambda$-quiddities over $\mathbb{Z}$ is
\[\{(1,1,1), (-1,-1,-1), (a,0,-a,0), (0,a,0,-a), a \in \mathbb{Z}-\{-1,1\}\}.\]

ii) The only irreducible $\lambda$-quiddities over $\mathbb{N}$ are $(1,1,1)$ and $(0,0,0,0)$.

\end{theorem}

\begin{theorem}
\label{26}

The set of irreducible $\lambda$-quiddities over $\mathbb{N}^{*}$ is : $\{(1,\ldots,1) \in (\mathbb{N}^{*})^{3n}, n \in \mathbb{N}^{*}\}$.

\end{theorem}

\begin{proof}

This is an easy consequence of the combinatorial description of the solutions of $(E_{\mathbb{N}^{*}})$ given in Theorem 1 of \cite{O}. For a detailed proof, see Proposition 4.2.1 of \cite{M2}. For the sake of completeness, we provide here a complete and purely algebraic proof of this result.
\\
\\Let $n \in \mathbb{N}^{*}$. $(1,\ldots,1) \in (\mathbb{N}^{*})^{3n}$ is a $\lambda$-quiddity since $M_{3n}(1,\ldots,1)=M_{3}(1,1,1)^{n}=(-Id)^{n}=\pm Id$. Suppose this solution is reducible. There exist two solutions of $(E_{\mathbb{N}^{*}})$, $(a_{1},\ldots,a_{l})$ and $(b_{1},\ldots,b_{m})$, such that $(1,\ldots,1)=(a_{1},\ldots,a_{l}) \oplus (b_{1},\ldots,b_{m})$. In particular, $1=a_{l}+b_{1} \geq 2$. This inequality is absurd. Hence, $(1,\ldots,1) \in (\mathbb{N}^{*})^{3n}$ is an irreducible $\lambda$-quiddity over $\mathbb{N}^{*}$.
\\
\\Let $(a_{1},\ldots,a_{n})$ be a $\lambda$-quiddity over $\mathbb{N}^{*}$ different from $(1,\ldots,1)$ (this condition implies $n \geq 4$). Since the set of solutions of $(E_{\mathbb{N^{*}}})$ is invariant by cyclic rotations, we can suppose $a_{1} \neq 1$. First, we prove that there exists $1 \leq i \leq n$ such that $a_{i}=1$. Suppose $a_{i}>1$ for all $i$. We consider the following recursive sequence $x_{0}=0$, $x_{1}=1$ and $x_{i+2}=a_{i}x_{i+1}-x_{i}$. We have $\left|x_{i+2}\right|\geq 2\left|x_{i+1}\right|-\left|x_{i}\right|$. Hence, by induction, $\left|x_{i+1}\right|>\left|x_{i}\right|$ for all $i \geq 0$. Moreover, we have $\begin{pmatrix}
    x_{n+2}  \\[4pt]
    x_{n+1}
    \end{pmatrix}=M_{n}(a_{1},\ldots,a_{n})\begin{pmatrix}
    x_{2}  \\[4pt]
    x_{1}
    \end{pmatrix}=\pm \begin{pmatrix}
    x_{2}  \\[4pt]
    x_{1}
    \end{pmatrix}$. Thus, $\left|x_{n+2}\right|=\left|x_{2}\right|$. This contradiction shows that there exists $1 \leq i \leq n$ such that $a_{i}=1$. Moreover, there is, at least, two $a_{i}$ different from 1. Indeed, suppose $a_{i}>1$ for all $i \geq 2$. We delete all the 3-tuples of consecutive 1's contained in $(a_{1},\ldots,a_{n})$. We obtain a $\lambda$-quiddity of the form $(a_{1})$, $(a_{1},1)$ or $(a_{1},1,1)$. These three tuples are not solutions. Hence, there is, at least, two $a_{i}$ different from 1. Now, we consider the maximal subsequence of consecutive 1's present in $(a_{1},\ldots,a_{n})$, that is to say the sequences $(a_{i},\ldots,a_{i+l-1})=(1,\ldots,1)$ with $a_{i-1},a_{i+l}>1$ (and $a_{n+1}=a_{1}$). The previous elements justify the existence of these sequences. We consider two cases.
\begin{itemize}
\item All the maximal subsequences of consecutive 1's have length $l \equiv 0,2 [3]$. Since $M_{4}(a,1,1,b)=-M_{1}(a+b-1)$ and $a+b-1>1$ if $a>1$, we can construct a $\lambda$-quiddity over $\mathbb{N}^{*}$ whose all maximal subsequences of consecutive 1's have length $l \equiv 0 [3]$. If we delete all these subsequences, we obtain a $\lambda$-quiddity over $\mathbb{N}^{*}$ without any 1. Hence, this case is impossible.
\item There exists a maximal subsequence of consecutive 1's of length $l \equiv 1 [3]$, $(a_{i},\ldots,a_{i+l-1})$ with $i \geq 2$ and $l \leq n-i+1$. We have the following equality :
\[(a_{1},\ldots,a_{n}) \sim \underbrace{(a_{i+l}-1,\ldots,a_{n},a_{1},\ldots,a_{i-1}-1)}_{l':=n-l} \oplus \underbrace{(1,\ldots,1)}_{l+2}.\]
\noindent Moreover, $l' \geq 2$. If $l'=2$ then we would have a $\lambda$-quiddity of length 2 containing a non zero element. Hence, $l' \geq 3$ and $(a_{1},\ldots,a_{n})$ is reducible.
\end{itemize}
	
\end{proof}

\noindent These two theorems give interesting information about irreducible $\lambda$-quiddities over infinite sets :
\begin{itemize}
\item There exists infinite rings such that $\ell_{A} \neq +\infty$.
\item It is possible to have a finite number of irreducible $\lambda$-quiddities over an infinite set.
\item There exists rings $R' \subset R$, such that $\ell_{R} \neq +\infty$ and $\ell_{R'}=+\infty$. 
\end{itemize}

\begin{theorem}[\cite{M4}, Théorème 2.8]
\label{27}

Let $I$ be a set containing at least two elements and $(A_{i})_{i \in I}$ a family of commutative and unitary rings. We suppose that at least two different $A_{i}$ have zero characteristic. Let $n \geq 3$. There exists an irreductible $\lambda$-quiddity of size $n$ over $\prod_{i \in I} A_{i}$.

\end{theorem}

\begin{examples}
{\rm The previous result gives : $\ell_{\mathbb{Z} \times \mathbb{Z}}=\ell_{\mathbb{Z}/2\mathbb{Z} \times \mathbb{Z}[X] \times \mathbb{Z}}=+\infty$.
}
\end{examples}

We also have information about finite rings. In particular, we have a precise classification of the irreducible solutions of $(E_{\mathbb{Z}/N\mathbb{Z}})$ in the cases $2 \leq N \leq 6$ (see \cite{M1} Théorème 2.5).

\begin{theorem}
\label{28}

i) (\cite{M1}, Théorème 2.5) The irreducible $\lambda$-quiddities over $\mathbb{Z}/4\mathbb{Z}$ are : $(\overline{1},\overline{1},\overline{1})$, $(\overline{-1},\overline{-1},\overline{-1})$, $(\overline{0},\overline{0},\overline{0},\overline{0})$, $(\overline{0},\overline{2},\overline{0},\overline{2})$ ,$(\overline{2},\overline{0},\overline{2},\overline{0})$ et $(\overline{2},\overline{2},\overline{2},\overline{2})$.
\\
\\ii) (\cite{M4}, Théorème 2.6) The irreducible $\lambda$-quiddities over $(\mathbb{Z}/2\mathbb{Z})\times (\mathbb{Z}/2\mathbb{Z})$ are (up to cyclic rotations):
\begin{itemize}
\item $((\overline{1},\overline{1}),(\overline{1},\overline{1}),(\overline{1},\overline{1}))$;
\item $((\overline{0},\overline{0}),(\overline{0},\overline{0}),(\overline{0},\overline{0}),(\overline{0},\overline{0})), ((\overline{0},\overline{0}),(\overline{0},\overline{1}),(\overline{0},\overline{0}),(\overline{0},\overline{1})), ((\overline{0},\overline{0}),(\overline{1},\overline{0}),(\overline{0},\overline{0}),(\overline{1},\overline{0})), 
\\((\overline{1},\overline{0}),(\overline{0},\overline{1}),(\overline{1},\overline{0}),(\overline{0},\overline{1}))$;
\item $((\overline{1},\overline{0}),(\overline{1},\overline{0}),(\overline{1},\overline{0}),(\overline{1},\overline{0}),(\overline{1},\overline{0}),(\overline{1},\overline{0})), ((\overline{0},\overline{1}),(\overline{0},\overline{1}),(\overline{0},\overline{1}),(\overline{0},\overline{1}),(\overline{0},\overline{1}),(\overline{0},\overline{1}))$.
\\
\end{itemize}

\noindent iii) (\cite{M4}, Théorème 2.7) The irreducible $\lambda$-quiddities over $\mathbb{F}_{4}:=\frac{(\mathbb{Z}/2\mathbb{Z})[X]}{<X^{2}+X+\overline{1}>}$ are (up to cyclic rotations):
\begin{itemize}
\item $(\overline{1},\overline{1},\overline{1})$;
\item $(\overline{0},\overline{0},\overline{0},\overline{0})$, $(\overline{0},X,\overline{0},X)$, $(\overline{0},X+\overline{1},\overline{0},X+\overline{1})$;
\item $(X,X,X,X,X)$, $(X+\overline{1},X+\overline{1},X+\overline{1},X+\overline{1},X+\overline{1})$;
\item $(X,X+\overline{1},X,X+\overline{1},X,X+\overline{1})$;
\item $(X,X,X+\overline{1},X+\overline{1},X,X,X+\overline{1},X+\overline{1})$;
\item $(X,X,X+\overline{1},X,X,X+\overline{1},X,X,X+\overline{1})$, $(X+\overline{1},X+\overline{1},X,X+\overline{1},X+\overline{1},X,X+\overline{1},X+\overline{1},X)$.
\end{itemize}

\end{theorem}

This result shows that finite rings with the same cardinality can have very different lists of irreducible $\lambda$-quiddities. In particular, $\ell_{A}$ does not only depend of $\left|A\right|$.

\subsection{Preliminary results}
\label{chap23}

The aim of this subsection is to collect several results which will be useful in the next section. We begin by considering the solutions of \eqref{p} of small size (see for example \cite{M3} Section 3.1) :

\begin{lemma}
\label{28}

\begin{itemize}
\item Equation \eqref{p} has no solution of size 1.
\item The 2-tuple $(0_{A},0_{A})$ is the only $\lambda$-quiddity over $A$ of size 2.
\item The two 3-tuples $(1_{A},1_{A},1_{A})$ and $(-1_{A},-1_{A},-1_{A})$ are the only $\lambda$-quiddities over $A$ of size 3. Besides, they are irreducible.
\item The set of solutions of \eqref{p} of size 4 is $\{(-a,b,a,-b) \in R^{4},~ab=0_{A}\} \cup \{(a,b,a,b) \in R^{4},~ab=2\times 1_{A}\}$. 
\item Solutions of $(E_{A})$ of size greater than 4 containing $\pm 1_{A}$ are reducible.
\item A solution of $(E_{A})$ of size 4 is irreducible if and only if it does not contain $\pm 1_{A}$.
\item Solutions of $(E_{A})$ of size greater than 5 containing $0_{A}$ are reducible.
\end{itemize}

\end{lemma}

\noindent The following result is an easy consequence of Lemma \ref{28}.

\begin{proposition}
\label{281}

Let $(G,+)$ be an infinite subgroup of an infinite ring $(A,+,\times)$. There are infinitely many irreducible $\lambda$-quiddities over $G$.

\end{proposition}

\begin{proof}

Since $(G,+)$ is a group, $0_{A} \in G$. For all $g \in G-\{-1_{A},1_{A}\}$, $(0_{A},g,0_{A},-g)$ is an irreducible $\lambda$-quiddity over $G$.

\end{proof}

Now, we will provide some details about the cardinality of the special linear group over a finite ring. We begin by recalling the following two classical results.

\begin{theorem}
\label{29}

Let $q$ a power of a prime number. $\left|SL_{2}(\mathbb{F}_{q})\right|=q(q^{2}-1)$.

\end{theorem}

\begin{proposition}
\label{210}

Let $n \in \mathbb{N}^{*}$ and $A_{1},\ldots,A_{n}$ $n$ commutative and unitary rings. $SL_{2}(A_{1} \times \ldots \times A_{n})$ is a group isomorphic to $SL_{2}(A_{1}) \times \ldots \times SL_{2}(A_{n})$.

\end{proposition}

\begin{proof}

We consider the following application :
\[\begin{array}{ccccc} 
\psi : & SL_{2}(A_{1} \times \ldots \times A_{n}) & \longrightarrow & SL_{2}(A_{1}) \times \ldots \times SL_{2}(A_{n}) \\
  &  \begin{pmatrix}
    (a_{i})_{1 \leq i \leq n} & (b_{i})_{1 \leq i \leq n} \\[4pt]
    (c_{i})_{1 \leq i \leq n} & (d_{i})_{1 \leq i \leq n}
     \end{pmatrix} & \longmapsto & \left(\begin{pmatrix}
    a_{i} & b_{i} \\[4pt]
    c_{i} & d_{i}
     \end{pmatrix}\right)_{1 \leq i \leq n}  \\
\end{array}.\]
\noindent This application is a group isomorphism.

\end{proof}

\begin{theorem}[\cite{X} 3.24, pages 211-215]
\label{211}

Let $N \geq 2$ and $p_{1},\ldots,p_{r}$ the distinct prime factors of $N$. 
\[\left|SL_{2}(\mathbb{Z}/N\mathbb{Z})\right|=N^{3}\prod_{i=1}^{r}\left(1-\frac{1}{p_{i}^{2}}\right).\]

\end{theorem}

We will also use an expression of the coefficients of $M_{n}(a_{1},\ldots,a_{n})$. For this, we define $K_{-1}=0$, $K_{0}=1$ and for $i \geq 1$ and $(a_{1},\ldots,a_{i}) \in A^{i}$ : \[K_{i}(a_{1},\ldots,a_{i})=
\left|
\begin{array}{cccccc}
a_1&1&&&\\[4pt]
1&a_{2}&1&&\\[4pt]
&\ddots&\ddots&\!\!\ddots&\\[4pt]
&&1&a_{i-1}&\!\!\!\!\!1\\[4pt]
&&&\!\!\!\!\!1&\!\!\!\!a_{i}
\end{array}
\right|.\] $K_{i}(a_{1},\ldots,a_{i})$ is the continuant of $a_{1},\ldots,a_{i}$. By induction, we have the following equality (see for example \cite{M2} Proposition 4.3.2) : 

\begin{proposition}
\label{212}

$(a_{1},\ldots,a_{n}) \in A^{n}$. $M_{n}(a_{1},\ldots,a_{n})=\begin{pmatrix}
    K_{n}(a_{1},\ldots,a_{n}) & -K_{n-1}(a_{2},\ldots,a_{n}) \\
    K_{n-1}(a_{1},\ldots,a_{n-1})  & -K_{n-2}(a_{2},\ldots,a_{n-1}) 
   \end{pmatrix}$.

\end{proposition}

\begin{lemma}
\label{213}

Let $A$ be a commutative and unitary ring. Let $n \geq 1$ and $(a_{1},\ldots,a_{n}) \in A^{n}$ such that $K_{n}(a_{1},\ldots,a_{n})=\epsilon \in \{\pm 1_{A}\}$. We define $x:=\epsilon K_{n-1}(a_{2},\ldots,a_{n})$ and $y:=\epsilon K_{n-1}(a_{1},\ldots,a_{n-1})$. We have the following equality :
\[M_{n+2}(x,a_{1},\ldots,a_{n},y)=-\epsilon Id.\]

\end{lemma}

\begin{proof}

Proposition \ref{212} gives : $M_{n+2}(x,a_{1},\ldots,a_{n},y)=\begin{pmatrix}
    K_{n+2}(x,a_{1},\ldots,a_{n},y) & -K_{n+1}(a_{1},\ldots,a_{n},y) \\
    K_{n+1}(x,a_{1},\ldots,a_{n})  & -K_{n}(a_{1},\ldots,a_{n}) 
	   \end{pmatrix}$.
\\
\\We have the two following equalities :
\begin{itemize}
\item $K_{n+1}(x,a_{1},\ldots,a_{n})=xK_{n}(a_{1},\ldots,a_{n})-K_{n-1}(a_{2},\ldots,a_{n})=\epsilon x-\epsilon x=0_{A}$; 
\item $K_{n+1}(a_{1},\ldots,a_{n},y)=yK_{n}(a_{1},\ldots,a_{n})-K_{n-1}(a_{1},\ldots,a_{n-1})=\epsilon y-\epsilon y=0_{A}$.
\end{itemize}
\noindent Since ${\rm det}(M_{n+2}(x,a_{1},\ldots,a_{n},y))=1_{A}$, $K_{n+2}(x,a_{1},\ldots,a_{n},y)=-\epsilon$. So, $M_{n+2}(x,a_{1},\ldots,a_{n},y)=-\epsilon Id$.

\end{proof}

We have seen that $R' \subset R$ does not imply in general $\ell_{R'} \leq \ell_{R}$. However, this is true if we only consider commutative and unitary rings.

\begin{proposition}
\label{214}

Let $A$ and $B$ be two commutative and unitary rings such that $A$ is a subring of $B$. 
\[\ell_{A} \leq \ell_{B}.\]

\end{proposition}

\begin{proof}

Let $n \in \mathbb{N}^{*}$ such that there exists $(a_{1},\ldots,a_{n}) \in A^{n}$ an irreducible $\lambda$-quiddity over $A$. Since $A \subset B$, $(a_{1},\ldots,a_{n})$ is a solution of $(E_{B})$. Suppose it is reducible over $B$. There exists $m,l \geq 3$, $(b_{1},\ldots,b_{m}) \in B^{m}$ and $(c_{1},\ldots,c_{l})$ a $\lambda$-quiddity over $B$ such that 
\[(a_{1},\ldots,a_{n}) \sim (b_{1},\ldots,b_{m}) \oplus (c_{1},\ldots,c_{l})=(b_{1}+c_{l},b_{2},\ldots,b_{m-1},b_{m}+c_{1},c_{2},\ldots,c_{l-1}).\]
\noindent In particular, $(c_{2},\ldots,c_{l-1}) \in A^{l-2}$ and $(b_{2},\ldots,b_{m-1}) \in A^{m-2}$. Besides, there exists $\alpha \in \{\pm 1_{B}\}=\{\pm 1_{A}\}$ such that $M_{l}(c_{1},\ldots,c_{l})=\alpha Id$. By Proposition \ref{212}, \[0_{B}=K_{l-1}(c_{1},\ldots,c_{l-1})=c_{1}K_{l-2}(c_{2},\ldots,c_{l-1})-K_{l-3}(c_{3},\ldots,c_{l-1})=-\alpha c_{1}-K_{l-3}(c_{3},\ldots,c_{l-1}).\]
\noindent Hence, $c_{1}=-\alpha K_{l-3}(c_{3},\ldots,c_{l-1})$. Since $(c_{3},\ldots,c_{l-1}) \in A^{l-3}$ and $A$ is a ring, $c_{1} \in A$. Similar argments show $c_{l} \in A$. Hence, $(c_{1},\ldots,c_{l})$ is a $\lambda$-quiddity over $A$. Besides, $(b_{1}+c_{l}, b_{m}+c_{1}) \in A^{2}$. Thus, $(b_{1},b_{m}) \in A^{2}$. So, $(a_{1},\ldots,a_{n})$ is reducible over $A$. Since this is not the case, $(a_{1},\ldots,a_{n})$ is irreducible over $B$ and $\ell_{B} \geq n$.

\end{proof}

\begin{remark}
{\rm There exists rings $A, B$ verifying $A \subset B$, $A \neq B$ and $\ell_{A}=\ell_{B}$. For instance, $\mathbb{Z} \subsetneq \mathbb{Z}[\pi]$ and $\ell_{\mathbb{Z}}=\ell_{\mathbb{Z}[\pi]}=4$ (\cite{M3} Théorème 2.7).
}
\end{remark}

\section{Lower and upper bounds for $\ell_{A}$}
\label{proof}

In a section dedicated to the well-known \og pigeon-hole principle \fg, M.\ Aigner and G.\ Ziegler say : \og Some mathematical principles, [\ldots], are so obvious that you might think they would only produce equally obvious results. \fg and, immediately after, they precise \og It ain't necessarily so \fg (see \cite{AZ} page 131). Reading these few lines gave us the idea used in what follows and the proof developed below provides an example which illustrates Aigner and Ziegler's words.

\subsection{Proof ot Theorem \ref{prin1}}
\label{chap31}

\noindent Let $(A,+,\times)$ be a finite commutative and unitary ring different from $\{0_{A}\}$ and $R$ a submagma of $(A,+)$.
\\
\\i) Since $A$ is finite, ${\rm car}(A)=p \neq 0$. Let $x \in R$ ($R \neq \emptyset$). Since $R$ is closed under $+$, $0_{A}=\sum_{i=1}^{p} x \in R$.
\\
\\We define $l:=\left|SL_{2}(A)\right|$, $l$ is well defined since $A$ is finite. Besides, $l \geq 4$. Let $n \in \mathbb{N}^{*}$ such that $n \geq l+1 \geq 5$ and $(a_{1},\ldots,a_{n}) \in R^{n}$. We consider the matrices $M_{k}(a_{1},\ldots,a_{k})$, with $1 \leq k \leq n$. We have $n>l$ objects from a set that contains only $l$ elements. Hence, by the pigeon-hole principle, there exists $1 \leq i<j \leq n$ such that :
\[M_{i}(a_{1},\ldots,a_{i})=M_{j}(a_{1},\ldots,a_{j})=M_{j-i}(a_{i+1},\ldots,a_{j})M_{i}(a_{1},\ldots,a_{i}).\]
\noindent Since $M_{i}(a_{1},\ldots,a_{i})$ is invertible (${\rm det}(M_{i}(a_{1},\ldots,a_{i}))=1_{A}$), we have :
\[M_{j-i}(a_{i+1},\ldots,a_{j})=Id.\]
\noindent Hence, $(a_{i+1},\ldots,a_{j})$ is a solution of \eqref{p} whose size is $1 \leq j-i \leq n-1$. We consider three cases :
\begin{itemize}
\item If $j-i \geq 3$ then we have following equality :
\[(a_{1},\ldots,a_{n}) \sim (a_{j},\ldots,a_{n},a_{1},\ldots,a_{j-1})=\underbrace{(0_{A},a_{j+1},\ldots,a_{n},a_{1},\ldots,a_{i},0_{A})}_{n-(j-i)+2 \geq 3} \oplus (a_{i+1},\ldots,a_{j}).\]
\noindent (if $j=n$ then $(a_{1},\ldots,a_{n}) \sim (a_{n},a_{1}\ldots,a_{n-1})=(0_{A},a_{1},\ldots,a_{i},0_{A}) \oplus (a_{i+1},\ldots,a_{n})$.)
\item If $j-i=2$ then, by Lemma \ref{28}, $a_{i+1}=a_{i+2}=0_{A}$ and ${\rm car}(A)=2$ and we have :
\[(a_{1},\ldots,a_{n}) \sim (a_{i+3},\ldots,a_{n},a_{1},\ldots,a_{i+2})=\underbrace{(a_{i+3},\ldots,a_{n},a_{1},\ldots,a_{i})}_{n-2 \geq3 } \oplus (0_{A},0_{A},0_{A},0_{A}).\]
\noindent (if $i=n-2$ then $(a_{1},\ldots,a_{n})=(a_{1},\ldots,a_{n-2}) \oplus (0_{A},0_{A},0_{A},0_{A})$.)
\item By Lemma \ref{28}, the case $j-i=1$ is impossible.
\end{itemize}

\noindent Hence, we can write $(a_{1},\ldots,a_{n})$ as a sum of an $u$-tuple with a solution of the equation \eqref{p} of size $v$ with $u, v \geq 3$. So, a $\lambda$-quiddity whose size is greater or equal to $l+1$ is reducible. Since $R$ is finite (as a subpart of a finite set), we conclude that the number of irreducible $\lambda$-quiddties over $R$ is finite.
\\
\\ii) Let $H:=\left\{\begin{pmatrix}
    1_{A} & x \\[4pt]
     0_{A} & 1_{A}
    \end{pmatrix},
\begin{pmatrix}
    -1_{A} & x \\[4pt]
     0_{A} & -1_{A}
    \end{pmatrix}, x \in A\right\}$. This set is a subgroup of $SL_{2}(A)$. We define $\left(SL_{2}(A)/H\right)_{d}$ as the set of right cosets of $H$. Let
\[l:=\left|\left(SL_{2}(A)/H\right)_{d}\right|=\frac{\left|SL_{2}(A)\right|}{\left|H\right|}=\left\{
    \begin{array}{ll}
        \frac{\left|SL_{2}(A)\right|}{\left|A\right|} & \mbox{if }~{\rm car}(A)=2; \\
        \frac{\left|SL_{2}(A)\right|}{2\left|A\right|} & \mbox{if }~{\rm car}(A) \neq 2.
    \end{array}
\right.\]

\noindent Let $n \in \mathbb{N}^{*}$ such that $n \geq l+3$ and $(a_{1},\ldots,a_{n}) \in A^{n}$. We consider the right cosets $H~M_{k}(a_{1},\ldots,a_{k})$, with $1 \leq k \leq n-2$. We have $n-2>l$ objects from a set that contains only $l$ elements. Hence, by the pigeon-hole principle, there exists $1 \leq i<j \leq n-2$ such that :
\[H~M_{i}(a_{1},\ldots,a_{i})=H~M_{j}(a_{1},\ldots,a_{j}).\]
\noindent Thus, $M=M_{j}(a_{1},\ldots,a_{j})(M_{i}(a_{1},\ldots,a_{i}))^{-1} \in H$. Moreover, we have the following equalities :
\begin{eqnarray*}
M &=& M_{j}(a_{1},\ldots,a_{j})(M_{i}(a_{1},\ldots,a_{i}))^{-1} \\
  &=& M_{j-i}(a_{i+1},\ldots,a_{j})M_{i}(a_{1},\ldots,a_{i})(M_{i}(a_{1},\ldots,a_{i}))^{-1} \\
	&=& M_{j-i}(a_{i+1},\ldots,a_{j}).
	\end{eqnarray*}

\noindent Hence, $M_{j-i}(a_{i+1},\ldots,a_{j}) \in H$. Thus, by Proposition \ref{212}, $K_{j-i}(a_{i+1},\ldots,a_{j})=\epsilon=\pm 1_{A}$. We define $x=\epsilon K_{j-i-1}(a_{i+2},\ldots,a_{j}) \in A$ and $y=\epsilon K_{j-i-1}(a_{i+1},\ldots,a_{j-1}) \in A$. By Lemma \ref{213}, $(x,a_{i+1},\ldots,a_{j},y)$ is a solution of $(E_{A})$. Its size is equal to $3 \leq j-i+2 \leq n-1$. We have the following equality :
\[(a_{1},\ldots,a_{n}) \sim (a_{j+1},\ldots,a_{n},a_{1},\ldots,a_{j})=(a_{j+1}-y,\ldots,a_{n},a_{1},\ldots,a_{i}-x) \oplus (x,a_{i+1},\ldots,a_{j},y).\]
\noindent Hence, we can write $(a_{1},\ldots,a_{n})$ as a sum of an $u$-tuple with a solution of $(E_{A})$ of size $v$ with $u, v \geq 3$. So, a $\lambda$-quiddity over $A$ whose size is greater or equal to $l+3$ is reducible. Hence, $\ell_{A} \leq l+2$.
\\
\\By Lemma \ref{28}, $(0_{A},0_{A},0_{A},0_{A})$ is an irreducible $\lambda$-quiddity over $A$. So $\ell_{A} \geq 4$. Since $A$ is finite, ${\rm car}(A)>0$. We suppose ${\rm car}(A) \geq 4$. To obtain the majoration $\ell_{A} \geq {\rm car}(A)$, we will adapt the proof of Theorem 2.6 of \cite{M1}. For all $m \in \mathbb{N}$, we define $m_{A}:=m \times 1_{A}$.
\\
\\By induction, we obtain, for all $m \in \mathbb{N}$, the following equality : $K_{m}(2_{A},\ldots,2_{A})=(m+1)_{A}$. Indeed, $K_{0}=1_{A}$ and $K_{1}(2_{A})=2_{A}$ and if we suppose there exits $m \geq 1$ such that $K_{m}(2_{A},\ldots,2_{A})=(m+1)_{A}$ and $K_{m-1}(2_{A},\ldots,2_{A})=m_{A}$ then 
\[K_{m+1}(2_{A},\ldots,2_{A})=2_{A}K_{m}(2_{A},\ldots,2_{A})-K_{m-1}(2_{A},\ldots,2_{A})=(2m+2)_{A}-m_{A}=(m+2)_{A}.\]
\noindent By induction, we obtain the desired formula.
\\
\\So, by Proposition \ref{212}, $M_{{\rm car}(A)}(2_{A},\ldots,2_{A})=Id$. Suppose this $\lambda$-quiddity is reducible. It exists $u,v \geq 3$, $(b_{1},\ldots,b_{u}) \in A^{u}$ and $(c_{1},\ldots,c_{v})$ a $\lambda$-quiddity over $A$ such that : \[(2_{A},\ldots,2_{A})=(b_{1},\ldots,b_{u}) \oplus (c_{1},\ldots,c_{v})=(b_{1}+c_{v},b_{2},\ldots,b_{u-1},b_{u}+c_{1},c_{2},\ldots,c_{v-1}).\]

\noindent Hence, $c_{2}=\ldots=c_{v-1}=2_{A}$. So, by Proposition \ref{212}, we have :
\[\pm 1_{A}=K_{v-2}(c_{2},\ldots,c_{v-1})=K_{v-2}(2_{A},\ldots,2_{A})=v_{A}-1_{A}.\]
\noindent Thus, ${\rm car}(A)$ divides $v$ or $(v-2)$. The two cases are impossible since $3 \leq v \leq {\rm car}(A)+2-u$.
\\
\\Hence, $(2_{A},\ldots,2_{A}) \in A^{{\rm car}(A)}$ is an irreducible solution of $(E_{A})$ and $\ell_{A} \geq {\rm car}(A)$.

\qed

\subsection{Applications}
\label{chap32}

If we combine the results of Theorem \ref{prin1} and the cardinalities recalled in Theorems \ref{29} and \ref{211}, we obtain Corollary \ref{prin2}. Moreover, we can modify the proof of Theorem \ref{prin1} i) to obtain the following result :

\begin{proposition}
\label{31}

Let $A$ be a finite commutative and unitary ring and $x \in \{-1_{A},0_{A},1_{A}\}$. There exists $k_{x} \in \mathbb{N}^{*}$ such that for all $n \geq k_{x}$ and for all $n$-tuple $(a_{1},\ldots,a_{n}) \in A^{n}$, there exists $1 \leq i \leq n$ and $1 \leq j \leq n-i+1$ such that $K_{i}(a_{j},\ldots,a_{j+i-1})=x$.

\end{proposition}

\begin{proof} 

We set $l:=\left|SL_{2}(A)\right|$, $l$ is well defined since $A$ is finite. Let $n \in \mathbb{N}^{*}$ such that $n \geq 3l+1$ and $(a_{1},\ldots,a_{n}) \in A^{n}$. We consider the matrices $M_{k}(a_{1},\ldots,a_{3k+1})$, with $0 \leq k \leq l$. We have $l+1>l$ objects from a set that contains only $l$ elements. Hence, by the pigeon-hole principle, there exists $0 \leq i<j \leq l$ such that :
\[M_{3i+1}(a_{1},\ldots,a_{3i+1})=M_{3j+1}(a_{1},\ldots,a_{3j+1})=M_{3(j-i)}(a_{3i+2},\ldots,a_{3j+1})M_{3i+1}(a_{1},\ldots,a_{3i+1}).\]
\noindent Since, $M_{3i+1}(a_{1},\ldots,a_{3i+1})$ is invertible, we obtain :
\[M_{3(j-i)}(a_{3i+2},\ldots,a_{3j+1})=Id.\]
\noindent So, by Proposition \ref{212}, we have the following equalities :
\begin{itemize}
\item $K_{3(j-i)}(a_{3i+2},\ldots,a_{3j+1})=1_{A}$;
\item $K_{3(j-i)-2}(a_{3i+3},\ldots,a_{3j})=-1_{A}$;
\item $K_{3(j-i)-1}(a_{3i+2},\ldots,a_{3j})=0_{A}$. 
\end{itemize}
\noindent Besides, $1 \leq 3(j-i)-2$ and $3(j-i) \leq n$.

\end{proof}

\begin{remark}
{\rm This result cannot be extended to another element of a finite commutative and unitary ring $A$. For instance, we can consider, for all $n \in \mathbb{N}^{*}$, $(0_{A},\ldots,0_{A}) \in A^{n}$. Indeed, we have :
\begin{itemize}
\item if $i \equiv \pm 1 [4]$, $K_{i}(0_{A},\ldots,0_{A})=0_{A}$;
\item if $i \equiv 2 [4]$, $K_{i}(0_{A},\ldots,0_{A})=-1_{A}$;
\item if $i \equiv 0 [4]$, $K_{i}(0_{A},\ldots,0_{A})=1_{A}$.
\end{itemize}
}
\end{remark}

\noindent In \cite{M4} Section 3.3, we have introduced the following open problem :

\begin{pro}

Let $I$ be a set. Find the necessary and sufficient conditions on $I$ such that there is a finite number of irreducible $\lambda$-quiddities over $\prod_{i \in I} \mathbb{Z}/2\mathbb{Z}$.

\end{pro}

Theorem \ref{prin1} gives us a sufficient condition. Indeed, if $I$ is a finite set, $A=\prod_{i \in I} \mathbb{Z}/2\mathbb{Z}$ is a finite ring, and so Theorem \ref{prin1} guarantees that the number of irreducible $\lambda$-quiddities over $A$ is finite. Moreover, this condition is also necessary. Indeed, if $I$ is an infinite set then $A=\prod_{i \in I} \mathbb{Z}/2\mathbb{Z}$ is an infinite ring, and the number of irreducible $\lambda$-quiddities over $A$ is infinite (see Proposition \ref{281}). Furthermore, thanks to Proposition \ref{210}, we have the following inequality : \[\ell_{\prod_{i=1}^{r} \mathbb{Z}/2\mathbb{Z}} \leq \frac{1}{2^{r}} \left(\prod_{i=1}^{r} 6\right)+2=3^{r}+2.\]

\noindent However, the following problem is still open.

\begin{pro}

Let $I$ be an infinite set and $A= \prod_{i \in I} \mathbb{Z}/2\mathbb{Z}$. Do we have $\ell_{A}=+\infty$ ? 

\end{pro}

We now give some numerical applications of Theorem \ref{prin1}. In the table below, we give the exact value of $\ell_{A}$ obtained with a computer (see \cite{CM} Section 5.2) and the value of $\theta_{A}=\left\{
    \begin{array}{ll}
        \frac{\left|SL_{2}(A)\right|}{\left|A\right|}+2 & \mbox{if }~{\rm car}(A)=2; \\
        \frac{\left|SL_{2}(A)\right|}{2\left|A\right|}+2 & \mbox{if }~{\rm car}(A) \neq 2.
    \end{array}
\right.$

\hfill\break 

\begin{center}
\begin{tabular}{|c|c|c|c|c|c|c|c|c|c|c|}
\hline
  $A$ & $\mathbb{Z}/7\mathbb{Z}$ & $\mathbb{Z}/8\mathbb{Z}$  & $\mathbb{Z}/9\mathbb{Z}$ & $\mathbb{Z}/10\mathbb{Z}$ & $\mathbb{Z}/11\mathbb{Z}$ & $\mathbb{Z}/12\mathbb{Z}$ & $\mathbb{Z}/13\mathbb{Z}$ & $\mathbb{Z}/14\mathbb{Z}$ & $\mathbb{Z}/15\mathbb{Z}$ & $\mathbb{Z}/16\mathbb{Z}$  \rule[-7pt]{0pt}{18pt} \\
  \hline
  $\ell_{A}$ & 9 & 8 & 12 & 12 & 19 & 15 & 25 & 20 & 26 & 24    \rule[-7pt]{0pt}{18pt} \\
	\hline
	  $\theta_{A}$ & 26 & 26 & 38 & 38 & 62 & 50 & 86 & 74 & 98 & 98  \rule[-7pt]{0pt}{18pt} \\
	\hline
	
\end{tabular}
\end{center}

\hfill\break 

\begin{center}
\begin{tabular}{|c|c|c|c|c|c|c|c|c|c|c|}
\hline
  $A$     & $\mathbb{F}_{8}$ & $\mathbb{Z}/2\mathbb{Z} \times \mathbb{F}_{4}$  & $\mathbb{Z}/2\mathbb{Z} \times \mathbb{Z}/2\mathbb{Z} \times \mathbb{Z}/2\mathbb{Z}$ &  $\mathbb{Z}/3\mathbb{Z} \times \mathbb{Z}/3\mathbb{Z}$ & $\mathbb{F}_{9}$ & $\mathbb{F}_{16}$ & $\mathbb{Z}/17\mathbb{Z}$ & $\mathbb{Z}/18\mathbb{Z}$ & $\mathbb{F}_{25}$ & $\mathbb{F}_{32}$  \rule[-7pt]{0pt}{18pt} \\
	\hline
	  $\theta_{A}$ & 65 & 47 & 29 & 34 & 42 & 257 & 146 & 110 & 314 & 1025    \rule[-7pt]{0pt}{18pt} \\
	\hline
	
\end{tabular}
\end{center}

\hfill\break 

\noindent Note that $\ell_{\mathbb{Z}/3\mathbb{Z} \times \mathbb{Z}/3\mathbb{Z}}=12$ (see \cite{M4} Section 4.1.3).

\subsection{A refinement of Theorem \ref{prin1} in the case $A=\mathbb{F}_{9}$}

The aim of this section is to find a better upper bound for $\ell_{\mathbb{F}_{9}}$. In the previous subpart, we have shown that $\ell_{\mathbb{F}_{9}} \leq 42$. Now, we will prove that $\ell_{\mathbb{F}_{9}} \leq 32$. To do this, we will follow the proof of Theorem \ref{prin1} ii), but we will choose another subgroup.

\begin{proposition}
\label{32}

$\ell_{\mathbb{F}_{9}} \leq 32$.

\end{proposition}

\begin{proof}

We define $H:=SL_{2}(\mathbb{F}_{3})$. This set is a subgroup of $SL_{2}(\mathbb{F}_{9})$. We define $\left(SL_{2}(\mathbb{F}_{9})/H\right)_{d}$ as the set of right cosets of $H$. Let $l:=\left|\left(SL_{2}(\mathbb{F}_{9})/SL_{2}(\mathbb{F}_{3})\right)_{d}\right|=30$.
\\
\\Let $n \in \mathbb{N}^{*}$ and $(a_{1},\ldots,a_{n}) \in \mathbb{F}_{9}^{n}$ such that $n \geq l+3$ and $(a_{1},\ldots,a_{n})$ is a $\lambda$-quiddity. We consider the right cosets $H~M_{k}(a_{1},\ldots,a_{k})$, with $1 \leq k \leq n-2$. We have $n-2>l$ objects from a set that contains only $l$ elements. Hence, by the pigeon-hole principle, there exists $1 \leq i<j \leq n-2$ such that :
\[H~M_{i}(a_{1},\ldots,a_{i})=H~M_{j}(a_{1},\ldots,a_{j}).\]
\noindent So, $M=M_{j}(a_{1},\ldots,a_{j})(M_{i}(a_{1},\ldots,a_{i}))^{-1} \in H$, that is to say $M_{j-i}(a_{i+1},\ldots,a_{j}) \in H$. 
\\
\\If $j-i=1$ then $a_{i+1} \in \mathbb{F}_{3}=\{\overline{0}, \overline{1},\overline{-1}\}$. By Lemma \ref{28}, $(a_{1},\ldots,a_{n})$ is reducible. 
\\
\\Now, we suppose $j-i \geq 2$. The elements of $H$ cannot have three coefficients equal to $\overline{0}$. Thus $K_{j-i}(a_{i+1},\ldots,a_{j})=\pm \overline{1}$ or $K_{j-i-1}(a_{i+2},\ldots,a_{j})=\pm \overline{1}$ or $K_{j-i-1}(a_{i+1},\ldots,a_{j-1})=\pm \overline{1}$. 
\\
\\We suppose $K_{j-i}(a_{i+1},\ldots,a_{j})=\epsilon=\pm \overline{1}$.
\\
\\We set $x=\epsilon K_{j-i-1}(a_{i+2},\ldots,a_{j}) \in \mathbb{F}_{9}$ and $y=\epsilon K_{j-i-1}(a_{i+1},\ldots,a_{j-1}) \in \mathbb{F}_{9}$. By Lemma \ref{213}, $(x,a_{i+1},\ldots,a_{j},y)$ is a solution of $(E_{\mathbb{F}_{9}})$. Its size is equal to $4 \leq j-i+2 \leq n-1$. We have the following equality :
\[(a_{1},\ldots,a_{n}) \sim (a_{j+1},\ldots,a_{n},a_{1},\ldots,a_{j})=\underbrace{(a_{j+1}-y,\ldots,a_{n},a_{1},\ldots,a_{i}-x)}_{n-(j-i) \geq 3} \oplus (x,a_{i+1},\ldots,a_{j},y).\]
\noindent Hence, $(a_{1},\ldots,a_{n})$ is reducible. 
\\
\\We suppose $K_{j-i-1}(a_{i+1},\ldots,a_{j-1})=\epsilon=\pm \overline{1}$.
\\
\\We set $x=\epsilon K_{j-i-2}(a_{i+2},\ldots,a_{j-1}) \in \mathbb{F}_{9}$ and $y=\epsilon K_{j-i-2}(a_{i+1},\ldots,a_{j-2}) \in \mathbb{F}_{9}$. By Lemma \ref{213}, $(x,a_{i+1},\ldots,a_{j-1},y)$ is a solution of $(E_{\mathbb{F}_{9}})$. Its size is equal to $3 \leq j-i+1 \leq n-2$. We have the following equality :
\[(a_{1},\ldots,a_{n}) \sim (a_{j},\ldots,a_{n},a_{1},\ldots,a_{j-1})=\underbrace{(a_{j}-y,\ldots,a_{n},a_{1},\ldots,a_{i}-x) }_{n-(j-i)+1 \geq 4}\oplus (x,a_{i+1},\ldots,a_{j-1},y).\]
\noindent Hence, $(a_{1},\ldots,a_{n})$ is reducible. 
\\
\\If $K_{j-i-1}(a_{i+2},\ldots,a_{j})=\epsilon=\pm \overline{1}$, the proof is the same as above.
\\
\\Hence, $\ell_{A} \leq l+2=32$.

\end{proof}

\subsection{A still open problem}

In most situations, we consider irreducible $\lambda$-quiddities over infinite subgroups (and even more precisely subgroups of $(\mathbb{C},+)$) or over finite rings. For the first case, there are many possible situations and generally we have to consider each case separately. For the second case, Theorem \ref{prin1} gives us a precise answer to the problem of the finiteness of the number of irreducible solutions. Hence, we could consider that the general question of the finiteness of the number of irreducible $\lambda$-quiddities over a fixed set is closed. However, there is still a case that deserves our attention. Indeed, Theorem \ref{prin1} does not solve the problem of the finiteness of the number of irreducible $\lambda$-quiddities over a finite submagma $R \neq \{0_{A}\}$ of an infinite ring $A$. Obviously, if there exists a finite subring $B$ of $A$ containing $R$ then Theorem \ref{prin1} give us the solution, but we can find several situations in which it is impossible to do such a reduction. 
\\
\\ \indent For instance, we choose two positive integers $N$ and $k$, with $N \geq 2$. We consider the two classical following sets $A:=(\mathbb{Z}/N\mathbb{Z})[X]$ and $R:=(\mathbb{Z}/N\mathbb{Z})_{k}[X]:=\{P \in A,~{\rm deg}(P) \leq k\}$. The ring $A$ is infinite and $R$ is a finite submagma of $A$. Besides, it does not exist a finite subring $B$ of $A$ containing $R$. Indeed, a ring containing $R$ necessarily contains $X^{l}$ for all $l \in \mathbb{N}$.
\\
\\ \indent In some cases, we can easily conclude. Indeed, consider for instance the sets $A:=(\mathbb{Z}/2\mathbb{Z})[X]$ and $R:=\{\overline{0},X\}$. The two sets $A$ and $R$ verify the desired conditions. For all $n \in \mathbb{N}^{*}$, $(X,\ldots,X) \in R^{n}$ is not a $\lambda$-quiddity over $R$ (since ${\rm deg}(K_{n}(X,\ldots,X))=n$). Hence, a $\lambda$-quiddity over $R$ necessarily contains $\overline{0}$. Thus, the only irreducible $\lambda$-quiddities over $R$ are $(\overline{0},\overline{0},\overline{0},\overline{0})$, $(\overline{0},X,\overline{0},X)$ and $(X,\overline{0},X,\overline{0})$. However, the general following problem remains open.

\begin{pro}

Let $R$ be a finite submagma of an infinite ring. Do we have $\ell_{R} <+\infty$ ? If the answer is no, can we characterise the submagma $R$ for which $\ell_{R} <+\infty$ ?

\end{pro}

\end{document}